\documentclass[11pt]{amsart}
\usepackage{amssymb}
\usepackage{graphicx}
\usepackage{cite}
\usepackage[left=3cm,right=3cm,top=3cm,bottom=3cm]{geometry}
\usepackage{color}
\usepackage{tikz-cd}
\usepackage{mathtools}
\usepackage[new]{old-arrows}
\usepackage{amsmath}
\usepackage{amsthm}
\usepackage{tikz}
\usepackage{amscd}
\usepackage{latexsym}
\usepackage{epsfig}
\usepackage{graphicx}
\usepackage{psfrag}
\usepackage{caption}
\usepackage{rotating}
\usepackage{mathtools}
\usepackage{enumitem}
\usepackage{bigints} 
\usepackage{longtable}
\usepackage[normalem]{ulem}

\allowdisplaybreaks[3]

\makeatletter
\@namedef{subjclassname@1991}{2020 Mathematics Subject Classification}
\makeatother

\newtheorem{theorem}{Theorem}
\newtheorem*{theorem*}{Theorem}
\newtheorem{proposition}[theorem]{Proposition}
\newtheorem{corollary}[theorem]{Corollary}

\theoremstyle{definition}
\newtheorem{definition}[theorem]{Definition}

\newtheorem{remark}[theorem]{Remark}
\newtheorem*{question*}{Question}

\newcommand{\sll}[1]{\mkern-4mu\mathbin{/\mkern-5mu/}_{\mkern-4mu{#1}}}

\newcommand{\g}{\mathfrak{g}}

\renewcommand{\exp}{\mathrm{exp}}

\numberwithin{equation}{section}

\title[Gelfand--Cetlin abelianizations of symplectic quotients]{Gelfand--Cetlin abelianizations of symplectic quotients}

\author[Peter Crooks]{Peter Crooks}
\author[Jonathan Weitsman]{Jonathan Weitsman}
\address[Peter Crooks]{Department of Mathematics and Statistics\\ Utah State University \\ 3900 Old Main Hill \\ Logan, UT 84322, USA}
\email{peter.crooks@usu.edu}
\address[Jonathan Weitsman]{Department of Mathematics \\ Northeastern University \\ 360 Huntington Avenue \\ Boston, MA 02115, USA}
\email{j.weitsman@northeastern.edu}

\begin{document}

\begin{abstract}           

We show that generic symplectic quotients of a Hamiltonian $G$-space $M$ by the action of a compact connected Lie group $G$ are also symplectic quotients of the same manifold $M$ by a compact torus. The torus action in question arises from certain integrable systems on $\g^*$, the dual of the Lie algebra of $G$. Examples of such integrable systems include the Gelfand--Cetlin systems of Guillemin--Sternberg in the case of unitary and special orthogonal groups, and certain integrable systems constructed for all compact connected Lie groups by Hoffman--Lane. Our abelianization result holds for smooth quotients, and more generally for quotients which are stratified symplectic spaces in the sense of Sjamaar--Lerman.

\end{abstract}

\maketitle
\begin{scriptsize}
\end{scriptsize}

\section{Introduction}

Let $G$ be a compact connected Lie group with Lie algebra $\g$. Suppose that $M$ is a Hamiltonian $G$-space, i.e. a symplectic manifold equipped with a symplectic action of $G$ and equivariant moment map $\mu : M \longrightarrow \g^*.$   The symplectic or Marsden--Weinstein \cite{MarsdenWeinstein} quotient of $M$ by $G$ at level $\xi\in\g^*$ is the topological space $$M\sll{\xi}G\coloneqq\mu^{-1}(\xi)/G_{\xi},$$ where $G_{\xi}\subset G$ is the $G$-stabilizer of $\xi.$  If $G$ acts freely on $\mu^{-1}(\xi),$ then $M\sll{\xi}G$ is a smooth symplectic manifold. In the absence of this freeness assumption, $M\sll{\xi}G$ is a stratified symplectic space in the sense of Sjamaar-Lerman \cite{SjamaarLerman}.

The purpose of this paper is to show that certain integrable systems on $\g^*$ allow us to express generic symplectic quotients of a Hamiltonian $G$-space $M$ as symplectic quotients of the same manifold 
$M$ by the action of a compact torus. Such integrable systems include the Gelfand--Cetlin systems constructed by Guillemin--Sternberg \cite{GuilleminSternbergGC,GuilleminSternbergThimm} 
for unitary and special orthogonal groups, as well Hoffman--Lane's more recent generalizations of Gelfand--Cetlin systems \cite{Lane} to arbitrary Lie type.

An example of our main result arises in classical mechanics \cite{GuilleminSternbergCollective}. Suppose that we are given a Hamiltonian $\operatorname{SO}(3)$-space $M$ with an invariant Hamiltonian function $H:M\longrightarrow\mathbb{R}$. The $\operatorname{SO}(3)$-action gives rise to two Poisson-commuting conserved quantities: the total angular momentum, and the angular momentum in some fixed direction in the Lie algebra of $\operatorname{SO}(3)$ corresponding to a choice of maximal torus. These quantities give the components of a moment map for a densely defined 2-torus action on $M,$ coming from the Gelfand--Cetlin system of Guillemin--Sternberg for the case of $\operatorname{SO}(3)$.\footnote{The square of the total angular momentum is a smooth function, but the orbits of its Hamiltonian flow do not have constant period, and so it does not generate a circle action.  Taking the square root gives a function whose Hamiltonian flow generates a circle action, but which is only continuous, not differentiable, at zero. It therefore does not define a Hamiltonian flow at zero.}  Our main result shows that for a non-zero value $\xi$ of the angular momentum, the symplectic quotient $M\sll{\xi}\operatorname{SO}(3)$ coincides with an appropriate symplectic quotient of $M$ under the densely defined torus action.  See \cite{GuilleminSternbergCollective} for more examples of these techniques.

\subsection{Main result}   We introduce the notion of a Gelfand--Cetlin datum $(\lambda_{\text{big}},\g^*_{\text{s-reg}})$ in Definition \ref{Definition: Main definition}. This amounts to $\lambda_{\text{big}}$ being a continuous map on $\g^*$ that restricts to a Poisson moment map for a Hamiltonian action of a compact torus $\mathbb{T}_{\text{big}}$ on an open dense subset $\g^*_{\text{s-reg}}\subset\g^*$, along with some extra conditions that capture salient properties of the classical Gelfand--Cetlin systems. One of these conditions is that the open, symplectic submanifold $M_{\text{s-reg}}\coloneqq\mu^{-1}(\g^*_{\text{s-reg}})\subset M$ be a Hamiltonian $\mathbb{T}_{\text{big}}$-space with moment map
$\lambda_M\coloneqq(\lambda_{\text{big}}\circ\mu)\big\vert_{M_{\text{s-reg}}}$ for any Hamiltonian $G$-space $M$ with moment map $\mu:M\longrightarrow\g^*$.
The results of Guillemin--Sternberg \cite{GuilleminSternbergGC,GuilleminSternbergThimm} imply that Gelfand--Cetlin data exist for all unitary and special orthogonal groups, while more recent results of Hoffman--Lane \cite{Lane} imply that such data exist in all Lie types.

The following is the main result of our paper.

\begin{theorem*}\label{Theorem: Main theorem}
Let $G$ be a compact connected Lie group, and $M$ a Hamiltonian $G$-space with moment map $\mu:M\longrightarrow\g^*$. Suppose that $(\lambda_{\emph{big}},\g^*_{\emph{s-reg}})$ is a Gelfand--Cetlin datum, and consider a point $\xi\in\g^*_{\emph{s-reg}}$.
\begin{itemize}
\item[\textup{(i)}] The torus $\mathbb{T}_{\emph{big}}$ acts freely on $\lambda_M^{-1}(\lambda_{\emph{big}}(\xi))$ if and only if $G_{\xi}$ acts freely on $\mu^{-1}(\xi)$. In this case, there is a canonical symplectomorphism $M\sll{\xi}G\cong M_{\emph{s-reg}}\sll{\lambda_{\emph{big}}(\xi)}\mathbb{T}_{\emph{big}}$.
\item[\textup{(ii)}] There is a canonical isomorphism $M\sll{\xi}G\cong M_{\emph{s-reg}}\sll{\lambda_{\emph{big}}(\xi)}\mathbb{T}_{\emph{big}}$ of stratified symplectic spaces. 
\end{itemize}
\end{theorem*}

Part (ii) is strictly more general than (i). Part (i) is included for the sake of exposition and accessibility.

One may regard this theorem as an approach to abelianizing the generic symplectic quotients of a Hamiltonian $G$-space $M$, i.e. to presenting such quotients as symplectic quotients by a compact torus. An alternative approach to abelianization is pursued in the work of Guillemin--Jeffrey--Sjamaar \cite{guillemin-jeffrey-sjamaar}.

 \subsection{Organization}
Section \ref{Section: Background and conventions} briefly establishes some of our conventions concerning Lie theory and Hamiltonian geometry. Section \ref{Section: Gelfand--Cetlin data} subsequently motivates and contextualizes the notion of a Gelfand--Cetlin datum. Our main result is then proved in Section \ref{Section: The abelianization theorem} for smooth quotients. A generalization to stratified symplectic spaces is formulated and proved in Section \ref{Section: Generalization to stratified symplectic spaces}.

\subsection*{Acknowledgements} The authors would like to thank Megumi Harada and Jeremy Lane for exceedingly useful conversations. P.C. acknowledges support from a Utah State University startup grant, while J.W. acknowledges support from Simons Collaboration Grant \# 579801.

\section{Background and conventions}\label{Section: Background and conventions} 
This section establishes some of our notation and conventions regarding Lie theory and Hamiltonian geometry. 
\subsection{Tori}
The Lie algebra of the unitary group $\operatorname{U}(1)$ is the real vector space $i\mathbb{R}\subset\mathbb{C}$ of purely imaginary numbers. We will identify this vector space with $\mathbb{R}$ in the obvious way. It follows that $\mathbb{R}^k$ is the Lie algebra of $\operatorname{U}(1)^{k}$ for all non-negative integers $k$, and that $$\mathbb{R}^k\longrightarrow\operatorname{U}(1)^k,\quad (x_1,\ldots,x_{k})\mapsto (e^{ix_1},\ldots,e^{ix_k})$$ is the exponential map for $\operatorname{U}(1)^k$. In certain contexts, we will implicitly use the dot product to regard $\mathbb{R}^k$ as the dual of the Lie algebra of $\operatorname{U}(1)^k$.  

\subsection{General compact connected Lie groups}\label{Subsection: Lie-theoretic preliminaries}
Let $G$ be a compact connected Lie group with Lie algebra $\g$ and rank $\ell$. One has the adjoint representations $\mathrm{Ad}:G\longrightarrow\operatorname{GL}(\g)$ and $\mathrm{ad}:\g\longrightarrow\mathfrak{gl}(\g)$, as well as the coadjoint representations $\mathrm{Ad}^*:G\longrightarrow\operatorname{GL}(\g^*)$ and $\mathrm{ad}^*:\g\longrightarrow\mathfrak{gl}(\g^*)$. The $G$-representations induce $G$-actions on $\g$ and $\g^*$, and thereby give rise to stabilizer subgroups
$$G_x\coloneqq\{g\in G:\mathrm{Ad}_g(x)=x\}\quad\text{and}\quad G_{\xi}\coloneqq\{g\in G:\mathrm{Ad}_g^*(\xi)=\xi\}$$ of $G$ for all  $x\in\g$ and $\xi\in\g^*$. On the other hand, the $\g$-representations allow us to define centralizers $$\g_x\coloneqq\{y\in\g:\mathrm{ad}_y(x)=0\}\quad\text{and}\quad\g_{\xi}\coloneqq\{y\in\g:\mathrm{ad}^*_y(\xi)=0\}$$ for all $x\in\g$ and $\xi\in\g^*$. It follows that $\g_x$ (resp. $\g_{\xi}$) is the Lie algebra of $G_x$ (resp. $G_{\xi}$). Let us also note that $\dim\g_x\geq\ell$ and $\dim\g_{\xi}\geq\ell$ for all $x\in\g$ and $\xi\in\g^*$. The regular loci $$\g_{\text{reg}}\coloneqq\{x\in\g:\dim\g_x=\ell\}\quad\text{and}\quad\g^*_{\text{reg}}\coloneqq\{\xi\in\g^*:\dim\g_{\xi}=\ell\}$$ are open, dense, $G$-invariant subsets of $\g$ and $\g^*$, respectively. An element $x\in\g$ (resp. $\xi\in\g^*$) then belongs to $\g_{\text{reg}}$ (resp. $\g_{\text{reg}}^*$) if and only if $\g_x$ (resp. $\g_{\xi}$) is a Cartan subalgebra of $\g$. This is equivalent to $G_x$ (resp. $G_{\xi}$) being a maximal torus of $G$. 

A few remarks on maximal tori and and Cartan subalgebras are warranted. Let $\mathfrak{t}\subset\g$ be a Cartan subalgebra, and write $T\subset G$ for the maximal torus with Lie algebra $\mathfrak{t}$. The exponential map $\exp:\g\longrightarrow G$ then restricts to a surjective homomorphism $\exp\big\vert_{\mathfrak{t}}:\mathfrak{t}\longrightarrow T$ of abelian groups. The kernel of the latter is a free $\mathbb{Z}$-submodule of $\mathfrak{t}$ with rank equal to $\ell$. It follows that the same is true of
$$\Lambda_{\mathfrak{t}}\coloneqq\frac{1}{2\pi}\mathrm{ker}\left(\exp\big\vert_{\mathfrak{t}}\right)\subset\mathfrak{t}.$$

\subsection{Hamiltonian geometry}
Let $(M,\sigma)$ be a Poisson manifold, i.e. $\sigma\in H^0(M,\Lambda^2TM)$ is a Poisson bivector field on the manifold $M$. Note that $\sigma$ may be regarded as a skew-symmetric bilnear map from two copies of $T^*M$ to the trivial rank-$1$ vector bundle over $M$. Contracting $\sigma$ with cotangent vectors in the first argument then determines a vector bundle morphism $\sigma^{\vee}:T^*M\longrightarrow TM$. One calls $(M,\sigma)$ \textit{non-degenerate} if $\sigma^{\vee}$ is an isomorphism. In this case, $(\sigma^{\vee})^{-1}=\omega^{\vee}$ for a unique symplectic form $\omega\in H^0(M,\Lambda^2 T^*M)$, where $\omega^{\vee}:TM\longrightarrow T^*M$ is the vector bundle morphism obtained by contracting $\omega$ with tangent vectors in the first argument. This process gives rise to a bijective correspondence between symplectic structures on $M$ and non-degenerate Poisson structures on $M$. We will thereby make no distinction between symplectic manifolds and non-degenerate Poisson manifolds. 

If $(M,\sigma)$ is a Poisson manifold, then $\sigma$ can be recovered from the Poisson bracket $\{\cdot,\cdot\}$ that it induces. This bracket associates to smooth functions $f_1,f_2:M\longrightarrow\mathbb{R}$ the smooth function
$$\{f_1,f_2\}\coloneqq\sigma(\mathrm{d}f_1\wedge\mathrm{d}f_2):M\longrightarrow\mathbb{R}.$$ At the same time, one defines the Hamiltonian vector field of a smooth function $f: M\longrightarrow\mathbb{R}$ by $X_f\coloneqq-\sigma^{\vee}(\mathrm{d}f)\in H^0(M,TM)$. It follows that $$\{f_1,f_2\}=-X_{f_1}(f_2)=X_{f_2}(f_1)$$ for all smooth functions $f_1,f_2:M\longrightarrow\mathbb{R}$.

Now let $G$ be a compact connected Lie group with Lie algebra $\g$ and exponential map $\exp:\g\longrightarrow G$. If $G$ acts smoothly on a manifold $M$, then each $\eta\in\g$ determines a generating vector field $\eta_M\in H^0(M,TM)$ by
$$(\eta_M)_m\coloneqq\frac{d}{dt}\bigg\vert_{t=0}\exp(-t\eta)\cdot m$$ for all $m\in M$. A Poisson manifold $(M,\sigma)$ with a smooth $G$-action will be called a \textit{Poisson Hamiltonian $G$-space} if $\sigma$ is $G$-invariant and $M$ comes equipped with a \textit{moment map}. This last term refers to $G$-equivariant smooth map $\mu: M\longrightarrow\g^*$ satisfying $X_{\mu^{\eta}}=\eta_M$ for all $\eta\in\g$, where $\mu^{\eta}:M\longrightarrow\mathbb{R}$ is the result of pairing $\mu$ with $\eta$ pointwise. We will reserve the term \textit{Hamiltonian $G$-space} for a Poisson Hamiltonian $G$-space whose underlying Poisson structure is symplectic.

It will be advantageous to recall the Poisson Hamiltonian $G$-space structure on $\g^*$. The Poisson bracket on $\g^*$ is given by
$$\{f_1,f_2\}(\xi)=\xi([\mathrm{d}_{\xi}f_1,\mathrm{d}_{\xi}f_2])$$ for all smooth functions $f_1,f_2:M\longrightarrow\mathbb{R}$ and points $\xi\in\g^*$, where $\mathrm{d}_{\xi}f_1,\mathrm{d}_{\xi}f_2\in(\g^*)^*=\g$ denote the differentials of $f_1,f_2$ at $\xi$, respectively. One finds that $\g^*$ is a Poisson Hamiltonian $G$-space with respect to the coadjoint action, and with the identity $\g^*\longrightarrow\g^*$ serving as the Poisson moment map.   

\section{Gelfand--Cetlin data}\label{Section: Gelfand--Cetlin data}
In this section, we define Gelfand--Cetlin data and introduce their main properties. This begins with the definition itself in \ref{Subsection: Definition}. The existence of Gelfand--Cetlin data is addressed in \ref{Subsection: Existence}, while concrete techniques for constructing such data are discussed in \ref{Subsection: Integrality} and \ref{Subsection: GC torus}. In \ref{Subsection: GS}, we describe concrete Gelfand--Cetlin data for unitary groups.  

\subsection{Definition and relation to integrable systems}\label{Subsection: Definition}
Let $G$ be a compact connected Lie group with Lie algebra $\g$ and rank $\ell$. Consider the quantities 
$$\mathrm{u}\coloneqq\frac{1}{2}(\dim\g-\ell)\quad\text{and}\quad\mathrm{b}\coloneqq\frac{1}{2}(\dim\g+\ell),$$ and introduce the following tori of \textit{small}, \textit{intermediate}, and \textit{big} ranks: $$\mathbb{T}_{\text{small}}\coloneqq\operatorname{U}(1)^{\ell},\quad\mathbb{T}_{\text{int}}\coloneqq\operatorname{U}(1)^{\mathrm{u}},\quad\text{and}\quad\mathbb{T}_{\text{big}}\coloneqq\mathbb{T}_{\text{small}}\times\mathbb{T}_{\text{int}}\cong\operatorname{U}(1)^{\mathrm{b}}.$$ The respective Lie algebras of these tori are
$$\mathbb{R}_{\text{small}}\coloneqq\mathbb{R}^{\ell},\quad\mathbb{R}_{\text{int}}\coloneqq\mathbb{R}^{\mathrm{u}},\quad\text{and}\quad\mathbb{R}_{\text{big}}\coloneqq\mathbb{R}_{\text{small}}\times\mathbb{R}_{\text{int}}\cong\mathbb{R}^{\mathrm{b}}.$$ 

\begin{definition}\label{Definition: Main definition}
A \textit{Gelfand--Cetlin datum} is a pair $(\lambda_{\text{big}},\g^*_{\text{s-reg}})$, consisting of a continuous map $\lambda_{\text{big}}=(\lambda_1,\ldots,\lambda_{\mathrm{b}}):\g^*\longrightarrow\mathbb{R}_{\text{big}}$ and open dense subset $\g^*_{\text{s-reg}}\subset\g^*$ that satisfy the following conditions:
\begin{itemize}
\item[\textup{(i)}] $\lambda_1,\ldots,\lambda_{\ell}$ are $G$-invariant on $\g^*$ and smooth on $\g^*_{\text{reg}}$;
\item[\textup{(ii)}] $\{\mathrm{d}_{\xi}\lambda_1,\ldots,\mathrm{d}_{\xi}\lambda_{\ell}\}$ is a $\mathbb{Z}$-basis of the lattice $\Lambda_{\g_{\xi}}\subset\g_{\xi}$ for all $\xi\in\g^*_{\text{reg}}$;
\item[\textup{(iii)}] $\g^*_{\text{s-reg}}\subset\g^*_{\text{reg}}$;
\item[\textup{(iv)}] $\lambda_{\text{big}}\big\vert_{\g^*_{\text{s-reg}}}:\g^*_{\text{s-reg}}\longrightarrow\mathbb{R}_{\text{big}}$ is a smooth submersion and moment map for a Poisson Hamiltonian $\mathbb{T}_{\text{big}}$-space structure on $\g^*_{\text{s-reg}}$;
\item[\textup{(v)}] $\lambda_{\text{big}}\big\vert_{\g^*_{\text{s-reg}}}:\g^*_{\text{s-reg}}\longrightarrow\lambda_{\text{big}}(\g^*_{\text{s-reg}})$ is a principal $\mathbb{T}_{\text{int}}$-bundle;
\item[\textup{(vi)}] if $M$ is a Hamiltonian $G$-space with moment map $\mu:M\longrightarrow\g^*$, then $$(\lambda_{\text{big}}\circ\mu)\big\vert_{\mu^{-1}(\g^*_{\text{s-reg}})}: \mu^{-1}(\g^*_{\text{s-reg}})\longrightarrow\mathbb{R}_{\text{big}}$$ is a moment map for a Hamiltonian $\mathbb{T}_{\text{big}}$-space structure on $\mu^{-1}(\g_{\text{s-reg}})$.
\end{itemize} 
In this case, we adopt the notation
$$\lambda_{\text{small}}\coloneqq (\lambda_1,\ldots,\lambda_{\ell}):\g^*\longrightarrow\mathbb{R}_{\text{small}},\quad\lambda_{\text{int}}\coloneqq(\lambda_{\ell+1},\ldots,\lambda_{\mathrm{b}}):\g^*\longrightarrow\mathbb{R}_{\text{int}},$$ $$M_{\text{s-reg}}\coloneqq\mu^{-1}(\g^*_{\text{s-reg}}),\quad\text{and}\quad\lambda_M\coloneqq(\lambda_{\text{big}}\circ\mu)\big\vert_{M_{\text{s-reg}}}:M_{\text{s-reg}}\longrightarrow\mathbb{R}_{\text{big}}.$$ We also refer to the elements of $\g^*_{\text{s-reg}}$ as the \textit{strongly regular} elements of $\g^*$.
\end{definition}

\begin{remark}
Condition (v) in Definition \ref{Definition: Main definition} is only slightly weaker than the existence of global action-angle coordinates on $\g^*_{\text{s-reg}}$. This existence question features prominently in \cite{Duistermaat,Lane}.
\end{remark}

It is instructive to consider this definition in relation to the theory of completely integrable systems. One is thereby led to the following result.

\begin{proposition}
Let $(\lambda_{\emph{big}},\g^*_{\emph{s-reg}})$ be a Gelfand--Cetlin datum. If $\mathcal{O}\subset\g^*$ is a coadjoint orbit and $\mathcal{O}_{\emph{s-reg}}\coloneqq\mathcal{O}\cap\g^*_{\emph{s-reg}}$, then $$\lambda_{\emph{int}}\big\vert_{\mathcal{O}_{\emph{s-reg}}}:\mathcal{O}_{\emph{s-reg}}\longrightarrow\lambda_{\emph{int}}(\mathcal{O}_{\emph{s-reg}})\subset\mathbb{R}_{\emph{int}}$$ is a completely integrable system, principal $\mathbb{T}_{\emph{int}}$-bundle, and moment map for a Hamiltonian action of $\mathbb{T}_{\emph{int}}$ on $\mathcal{O}_{\emph{s-reg}}$.
\end{proposition}

\begin{proof}
Note that the Hamiltonian vector field of any smooth function $\g^*\longrightarrow\mathbb{R}$ is tangent to $\mathcal{O}$. It follows that $\mathcal{O}_{\text{s-reg}}$ is stable under the action of $\mathbb{T}_{\text{big}}$ on $\g^*_{\text{s-reg}}$. Definition \ref{Definition: Main definition}(iv) now implies that $\mathcal{O}_{\text{s-reg}}$ is a Hamiltonian $\mathbb{T}_{\text{big}}$-space with moment map $\lambda_{\text{big}}\big\vert_{\mathcal{O}_{\text{s-reg}}}$. We conclude that $\lambda_{\text{int}}\big\vert_{\mathcal{O}_{\text{s-reg}}}$ is a moment map for the Hamiltonian action of $\mathbb{T}_{\text{int}}\subset\mathbb{T}_{\text{big}}$ on $\mathcal{O}_{\text{s-reg}}$. 

Since $\lambda_{\text{small}}$ is constant-valued on $\mathcal{O}$, Definition \ref{Definition: Main definition}(v) tells us that $\lambda_{\text{int}}\big\vert_{\mathcal{O}_{\text{s-reg}}}:\mathcal{O}_{\text{s-reg}}\longrightarrow\lambda_{\text{int}}(\mathcal{O}_{\text{s-reg}})$ is a principal $\mathbb{T}_{\text{int}}$-bundle. It therefore remains only to prove that $\dim\mathbb{T}_{\text{int}}=\frac{1}{2}\dim\mathcal{O}_{\text{s-reg}}$. This follows immediately from the fact that $\dim\mathbb{T}_{\text{int}}=\mathrm{u}=\frac{1}{2}(\dim\g-\ell)$.
\end{proof}

\subsection{Existence of Gelfand--Cetlin data}\label{Subsection: Existence}
It is natural to wonder about the generality in which Gelfand--Cetlin data exist. The earliest constructions are due to Guillemin--Sternberg \cite{GuilleminSternbergGC,GuilleminSternbergThimm}, and apply to all unitary groups $\operatorname{U}(n)$ and special orthogonal groups $\operatorname{SO}(n)$. The underlying techniques are based on Thimm's method, as described in \cite{GuilleminSternbergThimm}. Further details are outlined in Sections \ref{Subsection: Integrality}--\ref{Subsection: GS} of this paper. The case of symplectic groups is considerably more subtle and addressed in Harada's paper \cite{Harada}. 

Some recent work of Hoffman--Lane \cite{Lane} implies the existence of Gelfand--Cetlin data for an arbitrary compact connected Lie group $G$; the reader is referred to \cite[Section 6.2]{Lane} for the relevant details. This Hoffman--Lane paper is part of a broader program aimed at generalizing the results of Harada--Kaveh \cite{HaradaKaveh}.

\subsection{Construction of Gelfand--Cetlin data: integrality}\label{Subsection: Integrality}
We now discuss the construction of functions $\lambda_1,\ldots,\lambda_{\ell}:\g^*\longrightarrow\mathbb{R}$ satisfying Conditions (i) and (ii) in Definition \ref{Definition: Main definition}. Let $G$ be a compact connected Lie group with Lie algebra $\g$ and rank $\ell$. Choose a $G$-invariant inner product on $\g$, a Cartan subalgebra $\mathfrak{t}\subset\g$, and a closed, fundamental Weyl chamber $\mathfrak{t}_{+}\subset\mathfrak{t}$. The chamber $\mathfrak{t}_{+}$ is known to be a fundamental domain for the adjoint action of $G$ on $\g$. Our inner product thereby identifies $\mathfrak{t}_{+}$ with a fundamental domain $\mathfrak{t}_+^*\subset\g^*$ for the $G$-action on $\g^*$. We may therefore define a continuous surjection
$\pi:\g^*\longrightarrow\mathfrak{t}_+^*$ by the property that $(G\cdot\xi)\cap\mathfrak{t}_+^*=\{\pi(\xi)\}$ for all $\xi\in\g^*$, where $G\cdot\xi\subset\g^*$ is the coadjoint orbit of $\xi$. One sometimes calls $\pi$ the \textit{sweeping map} on $\g^*$ with respect to $\mathfrak{t}$; its fibers are exactly the coadjoint orbits of $G$, and $\pi(\g^*_{\text{reg}})$ is the interior $(\mathfrak{t}_+^*)^{\circ}$ of $\mathfrak{t}_+^*$. One also finds that the commutative diagram
$$\begin{tikzcd}[row sep=large,column sep=large]
\g^*_{\text{reg}}\arrow[hookrightarrow]{r} \arrow[d, swap, "\pi\big\vert_{\g^*_{\text{reg}}}"] & \g^* \arrow[d, "\pi"] \\
(\mathfrak{t}_+^*)^{\circ} \arrow[hookrightarrow]{r} & \mathfrak{t}_+^*
\end{tikzcd}$$ is Cartesian. The left vertical map $$\pi\big\vert_{\g^*_{\text{reg}}}:\g^*_{\text{reg}}\longrightarrow(\mathfrak{t}_+^*)^{\circ}$$ is a easily seen to be a smooth, surjective submersion.

Now choose a $\mathbb{Z}$-basis $\{\phi_1,\ldots,\phi_{\ell}\}$ of the $\mathbb{Z}$-submodule $\Lambda_{\mathfrak{t}}\subset\mathfrak{t}$. Note that the pairing between $\mathfrak{t}$ and $\mathfrak{t}^*$ allows one to regard $\phi_1,\ldots,\phi_{\ell}$ as functions on $\mathfrak{t}_+^*$. With this in mind, each $k\in\{1,\ldots,\ell\}$ determines a function $$\lambda_k\coloneqq\phi_k\circ\pi:\mathfrak{g}^*\longrightarrow\mathbb{R}.$$ The previous paragraph implies that $\lambda_1,\ldots,\lambda_{\ell}$ are smooth on $\g^*_{\text{reg}}$, while being $G$-invariant and continuous as functions on $\mathfrak{g}^*$. Given any $\xi\in\g^*_{\text{reg}}$, the differentials $\mathrm{d}_{\xi}\lambda_1,\ldots,\mathrm{d}_{\xi}\lambda_{\ell}\in(\g^*)^*=\g$ may be described as follows.

\begin{proposition}\label{Proposition: Lattice}
If $\xi\in\g_{\emph{reg}}^*$, then $\{\mathrm{d}_{\xi}\lambda_1,\ldots,\mathrm{d}_{\xi}\lambda_{\ell}\}$ is a $\mathbb{Z}$-basis of the lattice $\Lambda_{\g_{\xi}}\subset\g_{\xi}$.
\end{proposition}

\begin{proof}
Choose $g\in G$ for which $\mathrm{Ad}_g^*(\xi)\in(\mathfrak{t}_+^*)^{\circ}$. Since each function $\lambda_k$ is $G$-invariant, one has $$\mathrm{d}_{\xi}\lambda_k=\mathrm{d}_{\mathrm{Ad}_g^*(\xi)}\lambda_k\circ\mathrm{Ad}^*_g=\mathrm{Ad}_{g^{-1}}(\mathrm{d}_{\mathrm{Ad}_g^*(\xi)}\lambda_k)$$ for all $k\in\{1,\ldots,\ell\}$. We also observe that $\mathrm{Ad}_{g^{-1}}:\g\longrightarrow\g$ restricts to a $\mathbb{Z}$-module isomorphism $\lambda_{\mathfrak{t}}\overset{\cong}\longrightarrow\lambda_{\g_{\xi}}$.
It therefore suffices to take $\xi\in(\mathfrak{t}_+^*)^{\circ}$ and prove that $\mathrm{d}_{\xi}\lambda_k=\phi_k$ for all $k\in\{1,\ldots,\ell\}$.

Assume that $\xi\in(\mathfrak{t}_+^*)^{\circ}$, and fix $k\in\{1,\ldots,\ell\}$. Our invariant inner product allows us to regard $\mathfrak{t}^*$ as a subspace of $\g^*$. We then note that $\g^*=\mathfrak{t}^*\oplus T_{\xi}(G\cdot\xi)$, and that $T_{\xi}(G\cdot\xi)$ is contained in the kernel of $\mathrm{d}_{\xi}\lambda_k$. Let us also note that $T_{\xi}(G\cdot\xi)$ is the annihilator of $\mathfrak{t}$ in $\g^*$, and as such is contained in the kernel of $\phi_k$. These last two sentences reduce us to proving that $\mathrm{d}_{\xi}\lambda_k(\eta)=\phi_k(\eta)$ for all $\eta\in\mathfrak{t}^*$. On the other hand, we have $\mathrm{d}_{\xi}\lambda_k=\phi_k\circ\mathrm{d}_{\xi}\pi$. It therefore suffices to prove that $\mathrm{d}_{\xi}\pi(\eta)=\eta$ for all $\eta\in\mathfrak{t}^*$. But this is an immediate consequence of the following two observations: $(\mathfrak{t}_{+}^*)^{\circ}$ is an open subset of $\mathfrak{t}^*$, and $\pi(\eta)=\eta$ for all $\eta\in(\mathfrak{t}_{+}^*)^{\circ}$.
\end{proof}

\subsection{Construction of Gelfand--Cetlin data: Thimm's method}\label{Subsection: GC torus}
Retain the objects and notation discussed in Section \ref{Subsection: Integrality}. Let $G=G_0\supset G_1\supset\cdots\supset G_m$ be a descending filtration of $G$ by connected closed subgroups with respective Lie algebras $\g=\g_0\supset\g_1\supset\cdots\supset\g_m$. Let us also choose a Cartan subalgebra $\mathfrak{t}_j\subset\g_j$ and closed, fundamental Weyl chamber $(\mathfrak{t}_j)_{+}\subset\mathfrak{t}_j$ for each $j\in\{0,\ldots,m\}$. Our $G$-invariant inner product on $\g$ gives rise to a $G_j$-module isomorphism $\mathfrak{g}_j\cong\mathfrak{g}_j^*$, by means of which $\mathfrak{t}_j$ and $(\mathfrak{t}_j)_+$ correspond to subsets $\mathfrak{t}_j^*$ and $(\mathfrak{t}_j^*)_{+}$ of $\mathfrak{g}_j^*$. As in Section \ref{Subsection: Integrality}, one may define a continuous surjection $\pi_j:\mathfrak{g}_j^*\longrightarrow(\mathfrak{t}_j^*)_{+}$ by the condition that $(G_j\cdot\xi)\cap(\mathfrak{t}_j^*)_{+}=\{\pi_j(\xi)\}$ for all $\xi\in\mathfrak{g}_j^*$. 

Let $\ell=\ell_0\geq\ell_1\geq\cdots\geq\ell_m$ be the ranks of $G=G_0\supset G_1\supset\cdots\supset G_m$, respectively. Let us also choose a $\mathbb{Z}$-basis $\{\phi_{j1},\ldots,\phi_{j\ell_j}\}$ of the lattice $\Lambda_{\mathfrak{t}_j}\subset\mathfrak{t}_j$ for each $j\in\{0,\ldots,m\}$. As in Section \ref{Subsection: Integrality}, we define the functions 
$$\nu_{jk}\coloneqq\phi_{jk}\circ\pi_j:\g_j^*\longrightarrow\mathbb{R}$$ for $k\in\{1,\ldots,\ell_j\}$. The same section implies that $\nu_{j1},\ldots,\nu_{j\ell_j}$ are $G_j$-invariant and continuous on $\g_j^*$, as well as smooth on $(\mathfrak{g}_j^*)_{\text{reg}}$. We also have the following equivalent version of Proposition \ref{Proposition: Lattice}.

\begin{proposition}\label{Proposition: Lattice 2}
If $j\in\{0,\ldots,m\}$ and $\xi\in(\g_j^*)_{\emph{reg}}$, then $\{\mathrm{d}_{\xi}\nu_{j1},\ldots,\mathrm{d}_{\xi}\nu_{j\ell_j}\}$ is a $\mathbb{Z}$-basis of the lattice $\Lambda_{(\g_j)_{\xi}}\subset(\g_j)_{\xi}$.
\end{proposition}

Let $\sigma_j:\g^*\longrightarrow\g_j^*$ denote the transpose of the inclusion $\g_j\hookrightarrow\g$ for each $j\in\{0,\ldots,m\}$. Consider the functions on $\g^*$ defined by $$\lambda_{jk}\coloneqq\nu_{jk}\circ\sigma_j:\g^*\longrightarrow\mathbb{R}$$ for $j\in\{0,\ldots,m\}$ and $k\in\{1,\ldots,\ell_j\}$. It will be convenient to enumerate these functions as
\begin{equation}\label{Equation: Enumeration}\lambda_{\text{big}}\coloneqq(\lambda_1,\ldots,\lambda_{\mathrm{c}})\coloneqq(\lambda_{01},\ldots,\lambda_{0\ell},\lambda_{11},\ldots,\lambda_{1\ell_1}\ldots,\lambda_{m1},\ldots,\lambda_{m\ell_m}):\g^*\longrightarrow\mathbb{R}^{\mathrm{c}},\end{equation} where $\mathrm{c}\coloneqq\ell_0+\cdots+\ell_m$.

The discussion preceding Proposition \ref{Proposition: Lattice 2} implies that $\lambda_{\text{big}}$ is smooth on the open subset $$\mathcal{U}\coloneqq\bigcap_{j=0}^m\sigma_j^{-1}((\g_j^*)_{\text{reg}})\subset\g^*$$ Let us also define $$\g^*_{\text{s-reg}}\coloneqq\{\xi\in\mathcal{U}:\mathrm{d}_{\xi}\lambda_{\text{big}}\text{ is surjective}\}.$$ These last few sentences give context for the following consequence of \cite[Theorem 3.4]{GuilleminSternbergGC}.

\begin{proposition}\label{Proposition: Poisson moment map}
Let $\lambda_{\emph{big}}$ and $\g^*_{\emph{s-reg}}$ be as defined above.
\begin{itemize}
\item[\textup{(i)}] The restriction $\lambda_{\emph{big}}\big\vert_{\g^*_{\emph{s-reg}}}:\g^*_{\emph{s-reg}}\longrightarrow\mathbb{R}^{\mathrm{c}}$ is a moment map for a Poisson Hamiltonian $\operatorname{U}(1)^{\mathrm{c}}$-space structure on $\g^*_{\emph{s-reg}}$.
\item[\textup{(ii)}] If $M$ is a Hamiltonian $G$-space with moment map $\mu:M\longrightarrow\g^*$, then $$(\lambda_{\emph{big}}\circ\mu)\big\vert_{\mu^{-1}(\g^*_{\emph{s-reg}})}: \mu^{-1}(\g^*_{\emph{s-reg}})\longrightarrow\mathbb{R}^{\mathrm{c}}$$ is a moment map for a Hamiltonian $\operatorname{U}(1)^{\mathrm{c}}$-space structure on $\mu^{-1}(\g^*_{\emph{s-reg}})$.
\end{itemize}
\end{proposition}

This result has the following immediate connection to Definition \ref{Definition: Main definition}: the pair $(\lambda_{\text{big}},\g^*_{\text{s-reg}})$ is a Gelfand--Cetlin datum if and only if $\mathrm{c}=\mathrm{b}$, $\g^*_{\text{s-reg}}$ is dense in $\g^*$, and $\lambda_{\text{big}}\big\vert_{\g^*_{\text{s-reg}}}:\g^*_{\text{s-reg}}\longrightarrow\lambda_{\text{big}}(\g^*_{\text{s-reg}})$ is a principal $\mathbb{T}_{\text{int}}$-bundle. Guillemin--Sternberg \cite{GuilleminSternbergGC,GuilleminSternbergThimm} explicitly show these conditions to be achievable for $G=\operatorname{U}(n)$ and $G=\operatorname{SO}(n)$. Our next section outlines the details of the Guillemin--Sternberg construction for $G=\operatorname{U}(n)$. 

\subsection{Example of Gelfand--Cetlin datum: the Gelfand-Cetlin system on $\mathfrak{u}(n)^*$}\label{Subsection: GS}
Fix a positive integer $n$. Consider the Lie group $G\coloneqq\mathrm{U}(n)$ of unitary $n\times n$ matrices, and its Lie algebra $\g\coloneqq\mathfrak{u}(n)$ of skew-Hermitian $n\times n$ matrices. Let us also consider the real $\operatorname{U}(n)$-module $\mathcal{H}(n)$ of Hermitian $n\times n$ matrices. In what follows, we will freely identify $\mathfrak{u}(n)^*$ with $\mathcal{H}(n)$ by means of the non-degenerate, $G$-invariant bilinear form \begin{equation}\label{Equation: Pairing}\mathfrak{u}(n)\otimes_{\mathbb{R}}\mathcal{H}(n)\longrightarrow\mathbb{R},\quad\eta\otimes \xi\mapsto-i\mathrm{tr}(\eta\xi).\end{equation} 

Given an integer $j\in\{0,\ldots,n-1\}$, define the subgroup
$$G_j\coloneqq\left\{\left[\begin{array}{ c | c }
    I_{j} & 0 \\
    \hline
    0 & A
  \end{array}\right]:A\in\operatorname{U}(n-j)\right\}\subset G=\operatorname{U}(n).$$ The descending chain $\operatorname{U}(n)=G=G_0\supset G_1\supset\cdots\supset G_{n-1}$ then induces such a chain $\mathfrak{u}(n)=\g=\g_0\supset\g_1\supset\cdots\supset\g_{n-1}$ on the level of Lie algebras. We have
$$\g_j=\left\{\left[\begin{array}{ c | c }
    0 & 0 \\
    \hline
    0 & x
  \end{array}\right]:x\in\mathfrak{u}(n-j)\right\}\subset\mathfrak{u}(n)\quad\text{and}\quad\g_j^*=\left\{\left[\begin{array}{ c | c }
    0 & 0 \\
    \hline
    0 & \xi
  \end{array}\right]:\xi\in\mathcal{H}(n-j)\right\}\subset\mathcal{H}(n)$$ for all $j\in\{0,\ldots,n-1\}$, where the second equation implicitly uses \eqref{Equation: Pairing}.
The transpose $\sigma_j:\g^*\longrightarrow\g_j^*$ of $\g_j\subset\g$ then sends $\xi\in\g^*=\mathcal{H}(n)$ to the $(n-j)\times(n-j)$ submatrix in the bottom right-hand corner of $\xi$.

Now note that
$$\mathfrak{t}_j\coloneqq\left\{\left[\begin{array}{ c | ccc }
    0 & & 0 & \\
    \hline
     & ia_1 & &\\
    0 & & \ddots & \\
     & & & ia_{n-j}
  \end{array}\right]:a_1,\ldots,a_{n-j}\in\mathbb{R}\right\}\subset\g_j$$
is a Cartan subalgebra for each $j\in\{0,\ldots,n-1\}$. Let $\phi_{jk}\in\mathfrak{t}_j$ be the result of setting $a_k=1$ and $a_p=0$ for $p\neq k$, noting that $\{\phi_{j1},\ldots,\phi_{j(n-j)}\}$ is a $\mathbb{Z}$-basis of $\Lambda_{\mathfrak{t}_j}\subset\mathfrak{t}_j$. At the same time, consider the fundamental Weyl chamber
$$(\mathfrak{t}_j)_{+}\coloneqq\left\{\left[\begin{array}{ c | ccc }
    0 & & 0 & \\
    \hline
     & ia_1 & &\\
    0 & & \ddots & \\
     & & & ia_{n-j}
  \end{array}\right]:\let\scriptstyle\textstyle\substack{a_1,\ldots,a_{n-j}\in\mathbb{R}\\ a_1\geq\cdots\geq a_{n-j}}\right\}\subset\mathfrak{t}_j$$ for each $j\in\{0,\ldots,n-1\}$. Under the pairing \eqref{Equation: Pairing}, $(\mathfrak{t}_j)_{+}$ corresponds to the cone
$$(\mathfrak{t}^*_j)_{+}\coloneqq\left\{\left[\begin{array}{ c | ccc }
    0 & & 0 & \\
    \hline
     & a_1 & &\\
    0 & & \ddots & \\
     & & & a_{n-j}
  \end{array}\right]:\let\scriptstyle\textstyle\substack{a_1,\ldots,a_{n-j}\in\mathbb{R}\\ a_1\geq\cdots\geq a_{n-j}}\right\}\subset\mathcal{H}(n).$$ We then have sweeping maps $\pi_j:\g_j^*\longrightarrow(\mathfrak{t}^*_j)_{+}$ and compositions $\nu_{jk}\coloneqq\phi_{jk}\circ\pi_j:\g_j^*\longrightarrow\mathbb{R}$ for $j\in\{0,\ldots,n-1\}$ and $k\in\{1,\ldots,n-j\}$, as in Section \ref{Subsection: GC torus}. The functions
$$\lambda_{jk}\coloneqq\nu_{jk}\circ\sigma_j:\mathcal{H}(n)\longrightarrow\mathbb{R}$$ from Section \ref{Subsection: GC torus} are therefore given by the following condition: if $\xi\in\mathcal{H}(n)$ and $j\in\{0,\ldots,n-1\}$, then $\lambda_{j1}(\xi)\geq\lambda_{j2}(\xi)\geq\cdots\geq\lambda_{j(n-j)}(\xi)$ are the eigenvalues of the $(n-j)\times(n-j)$ submatrix in the bottom right-hand corner of $\xi$. 

Observe that the number of maps $\lambda_{jk}$ is $$n+(n-1)+\cdots+1=\frac{n(n+1)}{2}=\frac{1}{2}(\dim\mathfrak{u}(n)+n).$$ Our enumeration \eqref{Equation: Enumeration} therefore takes the form \begin{align*}\lambda_{\text{big}} & \coloneqq(\lambda_1,\ldots,\lambda_{\mathrm{\frac{n(n+1)}{2}}}) \\ & \coloneqq(\lambda_{01},\ldots,\lambda_{0n},\lambda_{11},\ldots,\lambda_{1(n-1)},\ldots,\lambda_{(n-2)1},\lambda_{(n-2)2},\lambda_{(n-1)1}):\mathcal{H}(n)\longrightarrow\mathbb{R}^{\frac{n(n+1)}{2}}.\end{align*} Let us also consider the open dense subset $$\mathcal{H}(n)_{\text{s-reg}}\coloneqq\left\{\xi\in\mathcal{H}(n):\let\scriptstyle\textstyle\substack{\lambda_{j1}(\xi)>\cdots>\lambda_{j(n-j)}(\xi)\text{ for all }j\in\{0,\ldots,n-1\}\\ \lambda_{0k}(\xi)>\cdots>\lambda_{(n-k)k}(\xi)\text{ for all }k\in\{1,\ldots,n\}}\right\}$$ of $\g^*=\mathcal{H}(n)$. By the paragraph following the proof of Proposition \ref{Proposition: Poisson moment map}, $(\lambda_{\text{big}},\mathcal{H}(n)_{\text{s-reg}})$ is a Gelfand--Cetlin datum if and only if $\lambda_{\text{big}}\big\vert_{\mathcal{H}(n)_{\text{s-reg}}}:\mathcal{H}(n)_{\text{s-reg}}\longrightarrow\lambda_{\text{big}}(\mathcal{H}(n)_{\text{s-reg}})$ is a principal bundle for $\mathbb{T}_{\text{int}}=\operatorname{U}(1)^{\frac{n(n-1)}{2}}$. This later condition is verified in \cite[Section 5]{GuilleminSternbergGC}.

\section{The abelianization theorem}\label{Section: The abelianization theorem}
This section is devoted to the proof of our abelianization theorem for smooth quotients. Some preliminary results are established in \ref{Subsection: Universal} and \ref{Subsection: Some supplementary results}, while the main proof appears in \ref{Subsection: Proof for smooth}.

\subsection{The universal maximal torus}\label{Subsection: Universal} 
Adopt the notation and conventions in Section \ref{Subsection: Definition}, and let $(\lambda_{\text{big}},\g^*_{\text{s-reg}})$ be a Gelfand--Cetlin datum. Given any $\xi\in\g^*_{\text{reg}}$, Definition \ref{Definition: Main definition}(ii) tells us that $\{\mathrm{d}_{\xi}\lambda_1,\ldots,\mathrm{d}_{\xi}\lambda_{\ell}\}$ is a basis of $\g_{\xi}$. This basis determines a vector space isomorphism $$\kappa_{\xi}:\g_{\xi}\overset{\cong}\longrightarrow\mathbb{R}^{\ell},\quad x_1\mathrm{d}_{\xi}\lambda_1+\cdots+x_{\ell}\mathrm{d}_{\xi}\lambda_{\ell}\mapsto (x_1,\ldots,x_{\ell}).$$ The torus $\mathbb{T}_{\text{small}}$ is then a universal maximal torus in the following sense.

\begin{proposition}\label{Proposition: Universal maximal torus}
If $\xi\in\g^*_{\emph{reg}}$, then $\kappa_{\xi}$ integrates to a Lie group isomorphism $\tau_{\xi}:G_{\xi}\overset{\cong}\longrightarrow\mathbb{T}_{\emph{small}}$.
\end{proposition}

\begin{proof}
It suffices to prove that $\kappa_{\xi}$ restricts to a $\mathbb{Z}$-module isomorphism from the kernel of $\exp\big\vert_{\g_{\xi}}:\g_{\xi}\longrightarrow G_{\xi}$ to $(2\pi\mathbb{Z})^{\ell}\subset\mathbb{R}^{\ell}$. This is an immediate consequence of Proposition \ref{Proposition: Lattice}.
\end{proof}

\begin{proposition}\label{Proposition: Alternative action description}
Let $M$ be a Hamiltonian $G$-space with moment map $\mu:M\longrightarrow\g^*$. Suppose that $\xi\in\g^*_{\emph{s-reg}}$. We then have $g\cdot m=\tau_{\xi}(g)\cdot m$ for all $g\in G_{\xi}$ and $m\in\mu^{-1}(\xi)$, where the left and right-hand sides denote the actions of $G_{\xi}\subset G$ on $M$ and $\mathbb{T}_{\emph{small}}\subset\mathbb{T}_{\emph{big}}$ on $M_{\emph{s-reg}}$, respectively.
\end{proposition}

\begin{proof}
Let $X_{\zeta}$ be the generating vector field on $M_{\text{s-reg}}$ determined by $\zeta\in\mathbb{R}_{\text{small}}$ via the action of $\mathbb{T}_{\text{small}}$ on $M_{\text{s-reg}}$. Write $Y_{\eta}$ for the generating vector field on $M$ determined by $\eta\in\g_{\xi}$ through the action of $G_{\xi}\subset G$ on $M$. It suffices to prove that $(X_{\kappa_{\xi}(\eta)})_m=(Y_{\eta})_m$ for all $\eta\in\g_{\xi}$ and $m\in\mu^{-1}(\xi)$. Setting $\gamma_j\coloneqq d_{\xi}\lambda_j$ and letting $e_j\in\mathbb{R}^{\ell}=\mathbb{R}_{\text{small}}$ denote the $j^{\text{th}}$ standard basis vector, this is equivalent to establishing that $(X_{e_j})_m=(Y_{\gamma_j})_m$ for all $j\in\{1,\ldots,\ell\}$ and $m\in\mu^{-1}(\xi)$. On the other hand, $Y_{\gamma_j}$ (resp. $X_{e_j}$) is the Hamiltonian vector field on $M$ (resp. $M_{\text{s-reg}}$) associated to $\mu^{\gamma_j}$ (resp. the $j^{\text{th}}$ component $\mu^*\lambda_j$ of $\lambda\circ\mu$). This further reduces us to proving that $\mathrm{d}_m\mu^{\gamma_j}=\mathrm{d}_m\mu^*\lambda_j$. But it is clear that
$$\mathrm{d}_m\mu^*\lambda_j=\mathrm{d}_\xi\lambda_j\circ\mathrm{d}_m\mu=\mathrm{d}_m\mu^{\gamma_j}$$ for all $m\in\mu^{-1}(\xi)$ and $j\in\{1,\ldots,\ell\}$.
\end{proof}

\subsection{Some supplementary results}\label{Subsection: Some supplementary results}
We now prove two supplementary facts needed to establish the main result of this paper. We continue with the notation and conventions of Sections \ref{Subsection: Definition} and \ref{Subsection: Universal}.

\begin{proposition}\label{Proposition: Two parts}
Let $M$ be a Hamiltonian $G$-space with moment map $\mu:M\longrightarrow\g^*$. Suppose that $\xi\in\g^*_{\emph{s-reg}}$.
\begin{itemize}
\item[\textup{(i)}] If $m\in\mu^{-1}(\xi)$ and $t\in\mathbb{T}_{\emph{int}}$ satisfy $t\cdot m\in\mu^{-1}(\xi)$, then $t=e$.
\item[\textup{(ii)}] The saturation of $\mu^{-1}(\xi)$ under the action of $\mathbb{T}_{\emph{int}}$ on $M_{\emph{s-reg}}$ is $\lambda_M^{-1}(\lambda_{\emph{big}}(\xi))$.
\end{itemize}
\end{proposition}

\begin{proof}
To verify (i), let $m\in\mu^{-1}(\xi)$ and $t\in\mathbb{T}_{\text{int}}$ be such that $t\cdot m\in\mu^{-1}(\xi)$. Let us also observe that $\mu$ is $\mathbb{T}_{\text{int}}$-equivariant when restricted to a map $M_{\text{s-reg}}\longrightarrow\g^*_{\text{s-reg}}$. These last two sentences imply that $\xi=t\cdot\xi$. Since $\mathbb{T}_{\text{int}}$ acts freely on $\g^*_{\text{s-reg}}$, we must have $t=e$.

We now verify (ii). To this end, note that $\lambda_M^{-1}(\lambda_{\text{big}}(\xi))$ is a $\mathbb{T}_{\text{int}}$-invariant subset of $M_{\text{s-reg}}$ that contains $\mu^{-1}(\xi)$. This implies that the saturation of $\mu^{-1}(\xi)$ is contained in $\lambda_M^{-1}(\lambda_{\text{big}}(\xi))$. For the opposite inclusion, suppose that $m\in\lambda_M^{-1}(\lambda_{\text{big}}(\xi))$. Definition \ref{Definition: Main definition}(v) tells us that $t\cdot\mu(m)=\xi$ for some $t\in\mathbb{T}_{\text{int}}$. By the equivariance property of $\mu$ mentioned in the previous paragraph, we must have $t\cdot m\in\mu^{-1}(\xi)$. This completes the proof of (ii). 
\end{proof}

Fix $\xi\in\g^*_{\text{s-reg}}$. In light of the previous proposition, we may define the map
$$\delta_{\xi}:\mu^{-1}(\xi)\times\mathbb{T}_{\text{int}}\longrightarrow\lambda_M^{-1}(\lambda_{\text{big}}(\xi)),\quad (m,t)\mapsto t\cdot m.$$ The following result is immediate consequence of the previous proposition.

\begin{corollary}\label{Corollary: Homeomorphism}
If $\xi\in\g^*_{\emph{s-reg}}$, then $\delta_{\xi}$ is a homeomorphism.
\end{corollary}

\subsection{Proof of the abelianization theorem}\label{Subsection: Proof for smooth}
Let us continue with the notation and conventions set in Sections \ref{Subsection: Definition}, \ref{Subsection: Universal}, and \ref{Subsection: Some supplementary results}. 

\begin{theorem}\label{Theorem: Free case}
Let $M$ be a Hamiltonian $G$-space with moment map $\mu:M\longrightarrow\g^*$. Suppose that $\xi\in\g^*_{\emph{s-reg}}$.
\begin{itemize}
\item[\textup{(i)}] The stabilizer $G_{\xi}$ acts freely on $\mu^{-1}(\xi)$ if and only if $\mathbb{T}_{\emph{big}}$ acts freely on $\lambda_M^{-1}(\lambda_{\emph{big}}(\xi))$.
\item[\textup{(ii)}]  In the case of \emph{(i)}, there is a canonical symplectomorphism $M\sll{\xi} G\cong M_{\emph{s-reg}}\sll{\lambda_{\text{big}}(\xi)}\mathbb{T}_{\emph{big}}$.
\end{itemize}
\end{theorem}

\begin{proof}
We begin by verifying (i). In light of Proposition \ref{Proposition: Universal maximal torus}, the multiplication map
$$\rho_{\xi}:G_{\xi}\times\mathbb{T}_{\text{int}}\longrightarrow\mathbb{T}_{\text{big}},\quad (g,t)\mapsto\tau_{\xi}(g)t$$ is a Lie group isomorphism.
We also note that the action of $G_{\xi}$ on $\mu^{-1}(\xi)$ and multiplication action of $\mathbb{T}_{\text{int}}$ on itself define an action $G_{\xi}\times\mathbb{T}_{\text{int}}$ on $\mu^{-1}(\xi)\times\mathbb{T}_{\text{int}}$. By Proposition \ref{Proposition: Alternative action description}, the homeomorphism $\delta_{\xi}$ is equivariant in the following sense:
$$\delta_{\xi}((g,t)\cdot x)=\rho_{\xi}(g,t)\cdot\delta_{\xi}(x)$$ for all $(g,t)\in G_{\xi}\times\mathbb{T}_{\text{int}}$ and $x\in\mu^{-1}(\xi)\times\mathbb{T}_{\text{int}}$. It follows that $G_{\xi}$ acts freely on $\mu^{-1}(\xi)$ if and only if $\mathbb{T}_{\text{big}}$ acts freely on $\lambda_M^{-1}(\lambda_{\text{big}}(\xi))$.

We now prove (ii). By Corollary \ref{Corollary: Homeomorphism}, the inclusion $\mu^{-1}(\xi)\longhookrightarrow\lambda_M^{-1}(\lambda_{\text{big}}(\xi))$ descends to a diffeomorphism $$\mathrm{f}:\mu^{-1}(\xi)\overset{\cong}\longrightarrow\lambda_M^{-1}(\lambda_{\text{big}}(\xi))/\mathbb{T}_{\text{int}}.$$ We also note that the $\mathbb{T}_{\text{big}}$-action on $\lambda_M^{-1}(\lambda_{\text{big}}(\xi))$ induces a residual action of the subtorus $\mathbb{T}_{\text{small}}$ on $\lambda_M^{-1}(\lambda_{\text{big}}(\xi))/\mathbb{T}_{\text{int}}$. Proposition \ref{Proposition: Alternative action description} then tells us that $$\mathrm{f}(g\cdot m)=\tau_{\xi}(g)\cdot \mathrm{f}(m)$$ for all $g\in G_{\xi}$ and $m\in\mu^{-1}(\xi)$. The map $\mathrm{f}$ therefore descends to a diffeomorphism
$$\varphi:M\sll{\xi}G\overset{\cong}\longrightarrow M_{\text{s-reg}}\sll{\lambda_{\text{big}}(\xi)}\mathbb{T}_{\text{big}}.$$ It therefore suffices to prove that $\varphi$ pulls the symplectic form $\beta$ on $M_{\text{s-reg}}\sll{\lambda_{\text{big}}(\xi)}\mathbb{T}_{\text{big}}$ back to the symplectic form $\alpha$ on $M\sll{\xi}G$. 

We have a commutative diagram $$\begin{tikzcd}
\mu^{-1}(\xi)\arrow[r, "\mathrm{j}"] \arrow[d, "\pi"'] & \lambda_M^{-1}(\lambda_{\text{big}}(\xi)) \arrow[d, "\theta"] \\
M\sll{\xi}G \arrow[r, swap, "\varphi"] & M_{\text{s-reg}}\sll{\lambda_{\text{big}}(\xi)}\mathbb{T}_{\text{big}}
\end{tikzcd},$$
where $\pi:\mu^{-1}(\xi)\longrightarrow\mu^{-1}(\xi)/G_{\xi}=M\sll{\xi}G$ and $\theta:\lambda_M^{-1}(\lambda_{\text{big}}(\xi))\longrightarrow\lambda_M^{-1}(\lambda_{\text{big}}(\xi))/\mathbb{T}_{\text{big}}=M_{\text{s-reg}}\sll{\lambda_{\text{big}}(\xi)}\mathbb{T}_{\text{big}}$ are the canonical quotient maps and $\mathrm{j}:\mu^{-1}(\xi)\longhookrightarrow\lambda_M^{-1}(\lambda_{\text{big}}(\xi))$ is the inclusion. We also have inclusion maps $\mathrm{k}:\mu^{-1}(\xi)\longhookrightarrow M$ and $\mathrm{l}:\lambda_M^{-1}(\lambda_{\text{big}}(\xi))\longhookrightarrow M$. Another consideration is that $\alpha$ (resp. $\beta$) is the unique $2$-form on $M\sll{\xi}G$ (resp. $M_{\text{s-reg}}\sll{\lambda_{\text{big}}(\xi)}\mathbb{T}_{\text{big}}$) for which $\pi^*\alpha=\mathrm{k}^*\omega$ (resp. $\theta^*\beta=\mathrm{l}^*\omega$), where $\omega$ is the symplectic form on $M$. It therefore suffices to prove that $\pi^*(\varphi^*\beta)=\mathrm{k}^*\omega$. On the other hand, our commutative diagram implies that
$$\pi^*(\varphi^*\beta)=\mathrm{j}^*(\theta^*\beta)=\mathrm{j}^*(\mathrm{l}^*\omega)=\mathrm{k}^*\omega.$$
This completes the proof.
\end{proof}

\section{Generalization to stratified symplectic spaces}\label{Section: Generalization to stratified symplectic spaces}
We now provide a generalization of Theorem \ref{Theorem: Free case} in the realm of stratified symplectic spaces \cite{SjamaarLerman}. In \ref{Subsection: Stratified symplectic}, we recall the immediately pertinent parts of Sjamaar and Lerman's more general theory of stratified symplectic spaces. The generalization of Theorem \ref{Theorem: Free case} to stratified symplectic spaces appears in \ref{Subsection: More general}. 

\subsection{Stratified symplectic spaces}\label{Subsection: Stratified symplectic}
Let $X$ be a topological space on which a compact torus $T$ acts continuously. Given a closed subgroup $H\subset T$, let $$X_H\coloneqq\{x\in X:T_x=H\}$$ be the locus of points with $T$-stabilizer $T_x$ equal to $H$. Denote by $\mathrm{Stab}(T,X)$ the set of all closed subgroups $H\subset T$ for which $X_H\neq\emptyset$.

Now let $G$ be a compact connected Lie group with Lie algebra $\g$. Suppose that $M$ is a Hamiltonian $G$-space with moment map $\mu:M\longrightarrow\g^*$. As discussed in the introduction to this paper, $M\sll{\xi}G$ is a stratified symplectic space \cite{SjamaarLerman} for all $\xi\in\g^*$. This means that $M\sll{\xi}G$ is naturally partitioned into symplectic manifolds satisfying certain compatibility conditions. While we refer the reader to \cite[Definition 1.12]{SjamaarLerman} for a precise definition and description of stratified symplectic spaces, the following exposition will be sufficient for our purposes. 

Fix a point $\xi\in\mathfrak{g}^*_{\text{reg}}$, and recall that $G_{\xi}\subset G$ is a maximal torus. Adopt the more parsimonious notation $$\mathrm{Stab}(G,\xi)\coloneqq\mathrm{Stab}(G_{\xi},\mu^{-1}(\xi)),$$ and note that $\mu^{-1}(\xi)$ is the disjoint union
$$\mu^{-1}(\xi)=\bigsqcup_{H\in\mathrm{Stab}(G,\xi)}\mu^{-1}(\xi)_H.$$
The arguments in the proof of \cite[Theorem 2.1]{SjamaarLerman} imply that each subset $\mu^{-1}(\xi)_H$ is a locally closed, $G_{\xi}$-invariant submanifold of $M$. These arguments also imply that the topological quotient $(\mu^{-1}(\xi)_H)/G_{\xi}$ carries a unique manifold structure for which the canonical map $\pi:\mu^{-1}(\xi)_H\longrightarrow(\mu^{-1}(\xi)_H)/G_{\xi}$ is a surjective submersion. One further consequence of \cite[Theorem 2.1]{SjamaarLerman} is the existence of a symplectic form $\overline{\omega}$ on $(\mu^{-1}(\xi)_H)/G_{\xi}$ such that $\pi^*\overline{\omega}$ is the pullback of $\omega$ along the inclusion $\mu^{-1}(\xi)_H\longhookrightarrow M$. It follows that $M\sll{\xi}G=\mu^{-1}(\xi)/G_{\xi}$ is a disjoint union
\begin{equation}\label{Equation: Symplectic strata}M\sll{\xi}G=\bigsqcup_{H\in\mathrm{Stab}(G,\xi)}(\mu^{-1}(\xi)_H)/G_{\xi}\end{equation} of symplectic manifolds, called the \textit{symplectic strata} of $M\sll{\xi}G$.

\begin{remark}\label{Remark: Strata}
The quotients $(\mu^{-1}(\xi)_H)/G_{\xi}$ need not be manifolds in the traditional sense of the term; each may have connected components of different dimensions. To obtain a stratification into genuine symplectic manifolds, one must refine \eqref{Equation: Symplectic strata} and declare the symplectic strata to be the connected components of the quotients $(\mu^{-1}(\xi)_H)/G_{\xi}$. The distinction between \eqref{Equation: Symplectic strata} and this refined stratification will not materially affect any argument in this paper.
\end{remark}

\begin{definition}\label{Definition: Isomorphism}
Let $G$ and $K$ be compact connected Lie groups with respective Lie algebras $\g$ and $\mathfrak{k}$. Suppose that $M$ (resp. $N$) is a Hamiltonian $G$-space (resp. Hamiltonian $K$-space) with moment map $\mu:M\longrightarrow\g^*$ (resp. $\nu:N\longrightarrow\mathfrak{k}^*$). Take $\xi\in\g^*_{\text{reg}}$ and $\eta\in\mathfrak{k}^*_{\text{reg}}$. A pair of maps $\varphi:M\sll{\xi}G\longrightarrow N\sll{\eta}K$ and $\phi:\mathrm{Stab}(G,\xi)\longrightarrow\mathrm{Stab}(K,\eta)$ will be called an \textit{isomorphism of stratified symplectic spaces} if the following conditions are satisfied:
\begin{itemize}
\item[\textup{(i)}] $\varphi$ is a homeomorphism;
\item[\textup{(ii)}] $\phi$ is a bijection;
\item[\textup{(iii)}] $\varphi$ restricts to a symplectomorphism $(\mu^{-1}(\xi)_H)/G_{\xi}\longrightarrow(\nu^{-1}(\eta)_{\phi(H)})/K_{\eta}$ for each $H\in\mathrm{Stab}(G,\xi)$. 
\end{itemize}
\end{definition}

\begin{remark}
Assume that this definition is satisfied. Equip $M\sll{\xi}G$ and $N\sll{\eta}K$ with the refined stratifications discussed in Remark \ref{Remark: Strata}. By (ii) and (iii), the association $S\mapsto\varphi(S)$ defines a bijection from the set of symplectic strata $S\subset M\sll{\xi}G$ to the set of symplectic strata in $N\sll{\eta}K$. Property (i) implies that this bijection is an isomorphism of partially ordered sets, i.e. any symplectic strata $S,T\subset M\sll{\xi}G$ satisfying $S\subset\overline{T}$ must also satisfy $\varphi(S)\subset\overline{\varphi(T)}$. We also know that $\varphi$ restricts to a symplectomorphism $S\longrightarrow\varphi(S)$ for all symplectic strata $S\subset M\sll{\xi}G$, as follows from (iii). In other words, an isomorphism in the sense of Definition \ref{Definition: Isomorphism} gives rise to an isomorphism between the refined symplectic stratifications on $M\sll{\xi}G$ and $N\sll{\eta}K$.  
\end{remark}

\subsection{A more general abelianization theorem}\label{Subsection: More general} 
Let us continue with the notation and conventions set in Section \ref{Section: The abelianization theorem}, as well as those in Section \ref{Subsection: Stratified symplectic} concerning stratified symplectic spaces. 

In preparation for our next proposition, we encourage the reader to recall Proposition \ref{Proposition: Universal maximal torus} and Corollary \ref{Corollary: Homeomorphism}.

\begin{proposition}\label{Proposition: Strata}
Let $M$ be a Hamiltonian $G$-space with moment map $\mu:M\longrightarrow\g^*$. Suppose that $\xi\in\g^*_{\emph{s-reg}}$.
\begin{itemize}
\item[\textup{(i)}] The association $H\mapsto\tau_{\xi}(H)$ defines a bijection $\mathrm{Stab}(G,\xi)\overset{\cong}\longrightarrow\mathrm{Stab}(\mathbb{T}_{\emph{big}},\lambda_{\emph{big}}(\xi))$.
\item[\textup{(ii)}] If $H\subset G_{\xi}$ is a closed subgroup, then $\delta_{\xi}$ restricts to a diffeomorphism
$$\mu^{-1}(\xi)_H\times\mathbb{T}_{\emph{int}}\overset{\cong}\longrightarrow\lambda_M^{-1}(\lambda_{\emph{big}}(\xi))_{\tau_{\xi}(H)}.$$
\end{itemize}
\end{proposition}

\begin{proof}
As in the proof of Theorem \ref{Theorem: Free case}(i), we have
$$\delta_{\xi}((g,t)\cdot x)=\rho_{\xi}(g,t)\cdot\delta_{\xi}(x)$$ for all $(g,t)\in G_{\xi}\times\mathbb{T}_{\text{int}}$ and $x\in\mu^{-1}(\xi)\times\mathbb{T}_{\text{int}}$. It follows that $K\mapsto\rho_{\xi}(K)$ defines a bijection $$\mathrm{Stab}(G_{\xi}\times\mathbb{T}_{\text{int}},\mu^{-1}(\xi)\times\mathbb{T}_{\text{int}})\overset{\cong}\longrightarrow\mathrm{Stab}(\mathbb{T}_{\text{big}},\lambda_{\text{big}}(\xi)),$$ and that $\delta_{\xi}$ restricts to a homeomorphism $$(\mu^{-1}(\xi)\times\mathbb{T}_{\text{int}})_K\overset{\cong}\longrightarrow\lambda_M^{-1}(\lambda_{\text{big}}(\xi))_{\rho_{\xi}(K)}$$ for all closed subgroups $K\subset G_{\xi}\times\mathbb{T}_{\text{int}}$. On the other hand, we clearly have a bijection
$$\mathrm{Stab}(G,\xi)\overset{\cong}\longrightarrow\mathrm{Stab}(G_{\xi}\times\mathbb{T}_{\text{int}},\mu^{-1}(\xi)\times\mathbb{T}_{\text{int}}),\quad H\mapsto H\times\{e\}\subset G_{\xi}\times\mathbb{T}_{\text{int}}.$$ We also note that $(\mu^{-1}(\xi)\times\mathbb{T}_{\text{int}})_K = \mu^{-1}(\xi)_H\times\mathbb{T}_{\text{int}}$ and $\rho_{\xi}(K)=\tau_{\xi}(H)$ for $K=H\times\{e\}$. These last three sentences combine to imply the desired results. 
\end{proof}

The following is our generalization of Theorem \ref{Theorem: Free case} to stratified symplectic spaces.

\begin{theorem}
If $\xi\in\g^*_{\emph{s-reg}}$, then there is a canonical isomorphism $M\sll{\xi} G\cong M_{\emph{s-reg}}\sll{\lambda_{\text{big}}(\xi)}\mathbb{T}_{\emph{big}}$ of stratified symplectic spaces.
\end{theorem}

\begin{proof}
By Corollary \ref{Corollary: Homeomorphism} and Proposition \ref{Proposition: Strata}, the inclusion $\mu^{-1}(\xi)\longhookrightarrow\lambda_M^{-1}(\lambda_{\text{big}}(\xi))$ descends to a homeomorphism $$\mathrm{f}:\mu^{-1}(\xi)\overset{\cong}\longrightarrow\lambda_M^{-1}(\lambda_{\text{big}}(\xi))/\mathbb{T}_{\text{int}}$$ whose restriction to $\mu^{-1}(\xi)_H$ is a diffeomorphism $$\mu^{-1}(\xi)_H\overset{\cong}\longrightarrow\lambda_M^{-1}(\lambda_{\text{big}}(\xi))_{\tau_{\xi}(H)}/\mathbb{T}_{\text{int}}$$ for all $H\in\mathrm{Stab}(G,\xi)$. We also note that the $\mathbb{T}_{\text{big}}$-action on $\lambda_M^{-1}(\lambda_{\text{big}}(\xi))$ induces a residual action of the subtorus $\mathbb{T}_{\text{small}}$ on $\lambda_M^{-1}(\lambda_{\text{big}}(\xi))/\mathbb{T}_{\text{int}}$. Proposition \ref{Proposition: Alternative action description} then tells us that $$\mathrm{f}(g\cdot m)=\tau_{\xi}(g)\cdot \mathrm{f}(m)$$ for all $g\in G_{\xi}$ and $m\in\mu^{-1}(\xi)$. The map $\mathrm{f}$ therefore descends to a homeomorphism
$$\varphi:M\sll{\xi}G\overset{\cong}\longrightarrow M_{\text{s-reg}}\sll{\lambda_{\text{big}}(\xi)}\mathbb{T}_{\text{big}}$$
whose restriction to $(\mu^{-1}(\xi)_H)/G_{\xi}$ is a diffeomorphism $$\varphi_H:(\mu^{-1}(\xi)_H)/G_{\xi}\overset{\cong}\longrightarrow(\lambda_M^{-1}(\lambda_{\text{big}}(\xi))_{\tau_{\xi}(H)})/\mathbb{T}_{\text{big}}$$ for all $H\in\mathrm{Stab}(G,\xi)$.

Now consider the bijection 
$$\phi:\mathrm{Stab}(G,\xi)\overset{\cong}\longrightarrow\mathrm{Stab}(\mathbb{T}_{\text{big}},\lambda_{\text{big}}(\xi)),\quad H\mapsto \tau_{\xi}(H)$$ from Proposition \ref{Proposition: Strata}(i). We claim that $\varphi$ and $\phi$ define an isomorphism of stratified symplectic spaces, in the sense of Definition \ref{Definition: Isomorphism}. In light of the previous paragraph, it suffices to prove the following for all $H\in\mathrm{Stab}(G,\xi)$: $\varphi_H$ pulls the symplectic form $\beta$ on $(\lambda_M^{-1}(\lambda_{\text{big}}(\xi))_{\tau_{\xi}(H)})/\mathbb{T}_{\text{big}}$ back to the symplectic form $\alpha$ on $(\mu^{-1}(\xi)_H)/G_{\xi}$. 

Proposition \ref{Proposition: Strata}(ii) implies that $\mu^{-1}(\xi)_H\subset\lambda_M^{-1}(\lambda_{\text{big}}(\xi))$. This leads to the commutative diagram $$\begin{tikzcd}
\mu^{-1}(\xi)_H\arrow[r, "\mathrm{j}"] \arrow[d, "\pi"'] & \lambda_M^{-1}(\lambda_{\text{big}}(\xi))_{\tau_{\xi}(H)} \arrow[d, "\theta"] \\
(\mu^{-1}(\xi)_H)/G_{\xi} \arrow[r, swap, "\varphi_H"] & (\lambda_M^{-1}(\lambda_{\text{big}}(\xi))_{\tau_{\xi}(H)})/\mathbb{T}_{\text{big}}
\end{tikzcd},$$
where $\pi:\mu^{-1}(\xi)_H\longrightarrow(\mu^{-1}(\xi)_H)/G_{\xi}$ and $\theta:\lambda_M^{-1}(\lambda_{\text{big}}(\xi))_{\tau_{\xi}(H)}\longrightarrow (\lambda_M^{-1}(\lambda_{\text{big}}(\xi))_{\tau_{\xi}(H)})/\mathbb{T}_{\text{big}}$ are the canonical quotient maps and $\mathrm{j}:\mu^{-1}(\xi)_H\longhookrightarrow\lambda_M^{-1}(\lambda_{\text{big}}(\xi))$ is the inclusion. We also have inclusion maps $\mathrm{k}:\mu^{-1}(\xi)\longhookrightarrow M$ and $\mathrm{l}:\lambda_M^{-1}(\lambda_{\text{big}}(\xi))\longhookrightarrow M$. Another consideration is that $\alpha$ (resp. $\beta$) is the unique $2$-form on $(\mu^{-1}(\xi)_H)/G_{\xi}$ (resp. $(\lambda_M^{-1}(\lambda_{\text{big}}(\xi))_{\tau_{\xi}(H)})/\mathbb{T}_{\text{big}}$) for which $\pi^*\alpha=\mathrm{k}^*\omega$ (resp. $\theta^*\beta=\mathrm{l}^*\omega$). It therefore suffices to prove that $\pi^*(\varphi_H^*\beta)=\mathrm{k}^*\omega$. On the other hand, our commutative diagram implies that
$$\pi^*(\varphi_H^*\beta)=\mathrm{j}^*(\theta^*\beta)=\mathrm{j}^*(\mathrm{l}^*\omega)=\mathrm{k}^*\omega.$$
This completes the proof.
\end{proof}

\bibliographystyle{acm} 
\bibliography{Abelianization}

\end{document}